\documentclass[dvipdfmx]{article}

\usepackage{amssymb,amsfonts,amsmath,amsthm,xcolor} 
\usepackage{latexsym,amscd}
\usepackage{url}
\usepackage{mathrsfs}
\usepackage[dvipdfmx]{graphicx}
\usepackage{cases}
\usepackage{diagbox}
\usepackage{tikz}

\newtheorem{theorem}{Theorem}[section]
\newtheorem{lemma}[theorem]{Lemma}
\newtheorem{corollary}[theorem]{Corollary}
\newtheorem{proposition}[theorem]{Proposition}

\newtheorem{definition}[theorem]{Definition}

\newtheorem{remark}[theorem]{Remark}

\numberwithin{equation}{section}

\def\a{\alpha}
\def\RR{\mathbb{R}}
\def\CC{\mathbb C}

\def\NN{\mathbb N}
\def\ZZ{\mathbb Z}

\def\R{\Re\mathfrak{e} \,}

\def\de{\delta}

\def\d{{\rm d}}

\def\vect#1{\mbox{\boldmath $#1$}} 
\def\absn#1{|#1|_n} 

\def\mutant{{non-congruent homometric space} }

\def\jm#1{\textcolor{black}{{#1}}}

\begin{document}

\title{Distinguishing regular polygons, cycle graphs, and circular metric spaces by the distance multiset and magnitude}

\author{Hiroki Kodama and 
Jun O'Hara\footnote{Supported by JSPS KAKENHI Grant Number 23K03083.}}
\maketitle

\begin{abstract}
We investigate how effectively finite metric spaces can be distinguished by distance-based invariants. As model spaces, we consider regular polygons, cycle graphs, and their generalization, circular metric spaces, and as invariants we consider the distance multiset, magnitude, and magnitude homology. We construct explicit families of homometric but non-congruent circular metric spaces, and in many even cases these examples also have the same magnitude as the original space. We prove that regular polygons are determined by the distance multiset among planar metric spaces, but not in general. We also determine, for several values of $n$, whether regular $n$-gons and $n$-cycle graphs are determined by magnitude.
\end{abstract}



\setcounter{tocdepth}{3}

\section{Introduction}

We investigate how effectively finite metric spaces can be distinguished by distance-based invariants. As model spaces, we consider regular polygons, cycle graphs, and their generalization, which we call circular spaces. As invariants, we focus on the multiset of pairwise distances and the magnitude. Our main objective is to compare the distinguishing power of these invariants by constructing explicit families of non-congruent spaces that they fail to distinguish.

The problem of reconstructing a finite metric space from the multiset of pairwise distances is known as the Unassigned Distance Geometry Problem (uDGP), or the homometric problem, and is a central topic in distance geometry (see Rosenblatt and Seymour \cite{RS}, Boutin and Kemper \cite{BK}). For specific classes of spaces, Skiena and collaborators (\cite{LSS}) studied the uDGP on the real line (the turnpike problem) and on the circle (the beltway problem). Our work fits naturally into this line of research.

The choice of regular polygons and cycle graphs is motivated by their role as discrete analogues of classical problems in geometry concerning the extent to which a space is determined by invariants. In convex geometry, Blaschke asked whether a convex body is determined by integrals of powers of chord lengths (the so-called moment integrals). In dimension two, objects with maximal symmetry, such as the disk, can be identified using the isoperimetric inequality. On the other hand, counterexamples arise by combining intermediate levels of symmetry, as in the examples of Mallows-Clark (\cite{MC}) and, in the non-convex setting, Caelli (\cite{Cae}). At the opposite extreme, Waksman showed that in the generic case (minimal symmetry), the space can be reconstructed (\cite{W}).

A regular polygon has high symmetry when its vertices are constrained to lie in the plane, but this is no longer the case once the ambient space is not restricted, and it is not generic either. In fact, whether a regular polygon is determined by the distance multiset depends on whether or not one imposes the planar constraint. Cycle graphs arise as natural counterparts of regular polygons when one restricts attention to graph metric spaces.

In addition to the distance multiset, we consider magnitude as a further invariant. The magnitude  is a kind of generalization of cardinality, where the count of elements is weighted by their mutual proximity. The precise definition will be given in Section 4. For finite metric spaces, the distance multiset is equivalent to the discrete Riesz energy. The authors' original motivation comes from the study of the Riesz energy function (also known as the Brylinski beta function) of manifolds and its residues. Magnitude was introduced by Leinster (\cite{L13}) as an application of category theory to finite metric spaces and was extended to manifolds by Meckes (\cite{Me}). \jm{Gimperlein and Goffeng (\cite{GG})}, and  Gimperlein, Goffeng and Louca (\cite{GGL}) later established a connection between magnitude and Riesz energy via a certain operator. Moreover, the first few coefficients in the asymptotic expansion of the magnitude of a manifold essentially coincide with the first few residues of the Riesz energy function, and both satisfy an inclusion-exclusion principle. Further details are given in Remark \ref{rem_Riesz_mag}.

Throughout the paper, we assume that the number of points $n$ in a finite metric space is at least four.

In Section 2, we construct, for circular spaces, explicit families of non-congruent spaces that are homometric to them. In particular, when $n\ge6$ is even, we construct such spaces so that the multisets of edge lengths incident to each vertex coincide with those of the original circular space; in this case, the magnitude also agrees. For regular $n$-gons, we further show that when $n\ge9$ is odd, one can construct homometric non-congruent spaces with the same property.

In Section 3, we study the reconstruction problem from the distance multiset. By the results of Section 2, circular spaces cannot in general be reconstructed from this invariant. However, under the additional assumption that all points lie in a common plane, we prove that regular polygons are determined by the distance multiset. Furthermore, when this assumption is removed, we show that for $n=4,5,6,7$, there exist spaces that are homometric but not congruent to a regular polygon and that admit isometric embeddings into Euclidean space.

In Section 4, we investigate the distinguishing power of magnitude and magnitude homology. By the results of Section 2, when $n\ge6$ is even, neither regular $n$-gons nor $n$-cycle graphs are determined by magnitude. On the other hand, we show that regular $n$-gons are determined by magnitude only for $n=4,5,7$, and that $n$-cycle graphs are determined by magnitude for $n=4,5$, \jm{and are not determined for even $n\ge6$}. Since magnitude can be interpreted as the Euler characteristic of magnitude homology, the latter is in principle a stronger invariant. Nevertheless, we prove that for regular $n$-gons, magnitude homology still fails to distinguish the space when $n=6$, \jm{$n=8$,} or $n\ge\jm{10}$.

\smallskip
We use the standard metric for the Euclidean spaces. 
By the regular polygon, which we denote by $R_n$,  we mean the set of vertices of the {\sl unit} regular polygon, i.e., the regular polygon with side length $1$. 
We denote the cycle graph with $n$ vertices by $C_n$. 

\smallskip
Acknowledgement: The authors thank deeply Yasuhiko Asao for the fruitful suggestions. 

\smallskip
Funding: This work was supported by JSPS KAKENHI Grant Number 23K03083. 

\section{Construction of homometric spaces}

Two finite metric spaces are said to be {\em homometric} if their distance multisets coincide. 
We give a way to construct non-congruent homometric spaces of regular polygons and cycle graphs. 

\begin{definition} \rm 
Define $|i|_n$ $(i\in\ZZ, \,n\in\NN)$ by $\displaystyle |i|_n=\min_{l\in\ZZ}|i-ln|$. 
\end{definition}

Note that $|\,\cdot\,|_n$ satisfies 
\begin{equation}\label{fcti}
\absn{i+ln}=\absn{-i}=\absn{i} \quad (l\in\ZZ).
\end{equation}

\begin{definition}\rm 
We say that a metric space $X=\{P_0,\dots,P_{n-1}\}$ is {\em circular of type} $(d_1,\dots,d_{\lfloor n/2\rfloor})$ $(0<d_1<\dots<d_{\lfloor n/2\rfloor})$ if $d(P_i,P_j)=d_{\absn{j-i}}$ for any $i,j$, where we put $d_0=0$. 
\end{definition}

Remark that the dihedral group $D_{n}$ acts isometrically on a circular space with $n$ vertices. 
The triangle inequality for $\triangle P_0 P_{i} P_{i+j}$, where the subscripts are considered modulo $n$, gives the ``{\em $n$-circular triangle inequality}''
\begin{equation}\label{triangle_inequality_delta}
d_{\absn{i}}+d_{\absn{j}}\ge d_{\absn{i+j}} \quad (\forall i,j).
\end{equation}

The cycle graph with $n$ vertices, $C_n$, is circular of type $(1,2,\dots,\lfloor n/2\rfloor)$, and the regular $n$-gon $R_n$ is circular of type $(\de_1, \dots, \de_{\lfloor n/2\rfloor})$, where 
$\de_i$ are given by 
\begin{equation}\label{delta}
\de_i=\frac{\sin({\pi i}/n)}{\sin({\pi}/n)}.
\end{equation}
The equality in the circular triangle inequality \eqref{triangle_inequality_delta} for $\de_i$ holds only in the trivial cases, $\absn{i}=0$ or $\absn{j}=0$. 

\begin{lemma}\label{lem_small_m-gon}
Suppose $0<d_1<\dots<d_{\lfloor n/2\rfloor}$ and $d_i$ satisfy the {$n$-circular triangle inequality} \eqref{triangle_inequality_delta}. 
{If $m<n$ then $d_1,\dots,d_{\lfloor m/2\rfloor}$ satisfy the $m$-circular triangle inequality. }
\end{lemma}

\begin{proof}
Since $n>m$, it holds that 
\[\absn{\a}=\min\{|\a|,n-|\a|\}\ge\min\{|\a|,m-|\a|\}=|\a|_m \quad (0\le\a\le m).\] 
Put $i'=j-i$ and $j'=k-j$. 
We may assume, by renumbering if necessary, that $0< i',j'\le \lfloor m/2\rfloor$. 
Then $d_{i'}+d_{j'}=d_{|i'|_n}+d_{|j'|_n}
\ge d_{|i'+j'|_n}\ge d_{|i'+j'|_m}$. 
\end{proof}

\begin{theorem}\label{thm_homometric}
A circular space $X$ has a non-congruent homometric space $X'$ if $\#X$  is greater than $3$. 
\end{theorem}

\begin{proof}
Put $n=\#X$. 
Suppose $X$ is circular of type $(d_1,\dots,d_{\lfloor n/2\rfloor})$, where $0<d_1<\dots<d_{\lfloor n/2\rfloor}$. 

When $n=4$, put $X'=\{A_1,A_2,A_3,A_4\}$ with 
$A_1A_2=A_1A_3=A_1A_4=A_3A_4=d_1$ and $A_2A_3=A_2A_4=d_2$ 
(Figure \ref{four_pts}, the right). 

\begin{figure}[htbp]
\begin{center}
\begin{minipage}{.45\linewidth}
\begin{center}
\includegraphics[width=4cm]{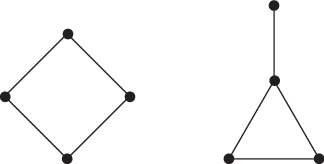}
\caption{Length $d_1$ edges}
\label{four_pts}
\end{center}
\end{minipage}
\hskip 0.4cm
\begin{minipage}{.45\linewidth}
\begin{center}
\includegraphics[width=4.8cm]{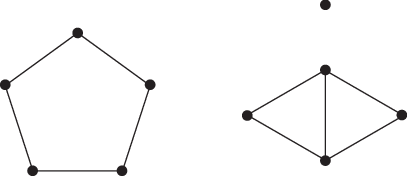}
\caption{Length $d_1$ edges}
\label{five_pts_bis}
\end{center}
\end{minipage}
\end{center}
\end{figure}

When $n=5$, put $X'=\{A_1,\dots, A_5\}$ with  
$A_2A_3=A_2A_4=A_2A_5=A_3A_4=A_3A_5=d_1$ and $A_1A_2=A_1A_3=A_1A_4=A_1A_5=A_4A_5=d_2$ 
(Figure \ref{five_pts_bis}).

\medskip
In the following, $d(A_i,B_j)$, etc. will be abbreviated as $A_iB_j$, etc. when there is no risk of confusion. 

\smallskip
(1) Assume $n=4k+1$ $(k\ge2)$. 
A $4k+1$-gon has $4k+1$ edges of length $d_1, d_2, \dots, d_{2k}$. 
Consider a $2k$-gon $\Delta_A=A_1\dots A_{2k}$ and a $2k+1$-gon $\Delta_B=B_1\dots B_{2k+1}$ with 
\[
A_iA_j=d_{|j-i|_{2k}}, \quad B_iB_j=d_{|j-i|_{2k+1}}.
\]
Then we have $4k+1$ edges of length $d_1, \dots, d_{k-1}$ and $3k+1$ edges of length $d_{k}$ so far. 
Put 
\[
A_iB_j=\left\{
\begin{array}{ll}
d_{k+j} &\>\>\mbox{ if }\>\> 1\le j\le k, \\[1mm]
d_{2k} &\>\>\mbox{ if }\>\> j=k+1 \>\>\mbox{ or }\>\>(i,j)=(2k,k+2), \\[1mm]
d_{2k-l} \>\> (1\le l\le k-1) &\>\>\mbox{ if }\>\>  j=k+l+1 \>\>\mbox{ and }\>\> i\le 2k-l, \>\>\mbox{ or } \\[1mm] 
 & \>\>\mbox{ \phantom{if} }\>\> j=k+l+2 \>\>\mbox{ and }\>\> i\ge 2k-l, \\[1mm]
d_k  &\>\>\mbox{ if }\>\> j=2k+1 \>\>\mbox{ and }\>\> i\le k
\end{array}
\right.
\]
(Table \ref{dAB4k+1}), and $\jm{X'}=\Delta_A\cup\Delta_B$. 
\begin{table}[htbp]
  \centering
  \begin{tabular}{|c||c|c|c|c|c|c|c|c|} \hline
    \diagbox{$B_j$}{$A_i$} & $1$ & \phantom{k+1} & $k$ & $k+1$ & $k+2$ & \phantom{k+1} & $2k-1$ & $2k$ \\ \hline\hline
    $1$ & $k+1$ & $\cdots$ & $\cdots$ & $\cdots$ & $\cdots$ & $\cdots$ & $\cdots$ & $k+1$ \\ \hline
    \vdots  &  \vdots   &  &  &  &  &  &  & \vdots \\ \hline
    $k$ & $2k$  & $\cdots$ & $\cdots$ & $\cdots$ & $\cdots$ & $\cdots$ & $\cdots$ & $2k$ \\ \hline
  $k+1$ & $2k$  & $\cdots$ & $\cdots$ & $\cdots$ & $\cdots$ & $\cdots$ & $\cdots$ & $2k$ \\ \hline
  $k+2$ & $2k-1$ &  $\cdots$ & $\cdots$ & $\cdots$ & $\cdots$ & $\cdots$ & $2k-1$ & $2k$ \\ \hline
  $k+3$ & $2k-2$ & $\cdots$ & $\cdots$ & $\cdots$ & $\cdots$ & $2k-2$ & $2k-1$ & $2k-1$ \\ \hline
  \vdots & \vdots &  &  &  & \mbox{\rotatebox{45}{$\cdots$}} &  &  & \vdots \\ \hline
   $2k$ & $k+1$ & $\cdots$ & $k+1$ & $k+1$ & $k+2$ & $\cdots$ & $\cdots$ & $k+2$ \\ \hline
 $2k+1$ & $k$ & $\cdots$ & $k$ & $k+1$ & $k+1$ & $\cdots$ & $\cdots$ & $k+1$ \\ \hline
  \end{tabular}
  \caption{\jm{Subscripts} of $d_\ast=A_iB_j$}
  \label{dAB4k+1}
\end{table}

Assuming that $\jm{X'}$ is a metric space, it is not isometric \jm{to} $X$ since $\jm{X'}$ has a cycle consisting of $2k$ edges with length $d_1$, whereas $X$ does not. 
To show that this is a \mutant of $X$, it remains to verify that the triangle inequality is satisfied by any triangle. 

First, Lemma \ref{lem_small_m-gon} implies that if all the vertices of a triangle belong to $\Delta_A$ (or $\Delta_B$) then the triangle inequality is satisfied. 

Suppose two vertices belong to $\Delta_A$ and the rest to $\Delta_B$, say $\triangle A_iA_jB_l$. 
Since $A_iA_j\le d_k\le A_iB_l, A_jB_l$, we have only to show $|A_iB_l-A_jB_l|\le A_iA_j$, which holds as $|A_iB_l-A_jB_l|\le d_1\le A_iA_j$. 
This is because the subscripts of $d_\ast=A_iB_l$ and $d_\ast=A_jB_l$ 
either coincide or differ by one. 

Suppose two vertices belong to $\Delta_B$ and the rest to $\Delta_A$, say $\triangle B_iB_jA_l$. 
We have only to show $|A_lB_i-A_lB_j|\le B_iB_j$. 
Suppose $A_lB_i=d_\lambda$ and $A_lB_j=d_\mu$. 
We show 
\begin{equation}\label{lambda_mu}
|\lambda-\mu|\le|j-i|_{2k+1}. 
\end{equation}

(i) If $1\le i,j\le k$, then for any $l$, $\lambda=k+i, \mu=k+j$. Since $|j-i|<k$ we have $|\lambda-\mu|=|j-i|=|j-i|_{2k+1}$. 

(ii) Suppose $1\le l\le k$ and $k+1\le i,j\le 2k+1$. Then $\lambda=3k+1-i$ and $\mu=3k+1-j$, which implies $|\lambda-\mu|=|j-i|=|j-i|_{2k+1}$ as $|j-i|\le k$. 

(iii) Suppose $1\le l\le k$ and $i\le k, \, k+1\le j\le 2k+1$. Then $\lambda=k+i$ and $\mu=3k+1-j$. 
We have 
\begin{equation}\label{la-mu}
\begin{array}{rll}
\mu-\lambda=2k+1-j-i&\le j-i-1 &\>\>\mbox{ since }\>\>k+1\le j,\\
\mu-\lambda=2k+1-j-i&\le 2k+1-(j-i)-1 &\>\>\mbox{ since }\>\>i\ge 1,\\
\lambda-\mu=j+i-2k-1&\le j-i-1 &\>\>\mbox{ since }\>\>i\le k,\\
\lambda-\mu=j+i-2k-1&\le 2k+1-(j-i) &\>\>\mbox{ since }\>\>j\le 2k+1,
\end{array}
\end{equation}
which implies 
\[
\mu-\lambda, \lambda-\mu\le\min\{j-i, 2k+1-(j-i)\}=|j-i|_{2k+1}.
\]

(iv) Suppose $k+1\le l\le 2k$ and $k+1\le i,j\le 2k+1$. Then $|\lambda-\mu|=|j-i|$ or $|j-i|-1$, which implies \eqref{lambda_mu}. 

(v) Suppose $k+1\le l\le 2k$ and $i\le k, \, k+1\le j\le 2k+1$. Then $\lambda=k+i$ and $\mu=3k+1-j$ or $3k+2-j$, and therefore $\left||\lambda-\mu|-|j+i-2k-1|\right|\le1$. 
\eqref{la-mu} shows that $|j+i-2k-1|\le|j-i|_{2k+1}-1$ holds unless $i\le k$ and $j=2k+1$, in which case $\lambda=k+i$ and $\mu=k+1$ and therefore $|\lambda-\mu|=i-1\le |i|_{2k+1}=|j-i|_{2k+1}$, whereas when $|j+i-2k-1|\le|j-i|_{2k+1}-1$, $\left||\lambda-\mu|-|j+i-2k-1|\right|\le1$ implies $|\lambda-\mu|\le|j-i|_{2k+1}$. 

Since $|\lambda-\mu|\le k$ as $k\le \lambda,\mu\le 2k$, \eqref{lambda_mu} implies
\[
|A_lB_i-A_lB_j|=|d_\lambda-d_\mu|\le d_{|\lambda-\mu|}\le d_{|j-i|_{2k+1}}=B_iB_j, 
\]
which completes the proof of the case when $n=4k+1$.

\medskip
(2) Assume $n=4k+3$ $(k\ge1)$. 
This case can be proved in the same way. 
Consider a $2k+1$-gon $\Delta_A=A_1\dots A_{2k+1}$ and a $2k+2$-gon $\Delta_B=B_1\dots B_{2k+2}$ with 
\[
A_iA_j=d_{|j-i|_{2k+1}}, \quad B_iB_j=d_{|j-i|_{2k+2}}.
\]
Put 
\[
A_iB_j=\left\{
\begin{array}{ll}
d_{k+j} &\>\>\mbox{ if }\>\> 1\le j\le k+1, \\[1mm]
d_{2k+1} &\>\>\mbox{ if }\>\> j=k+2 \>\>\mbox{ or }\>\>(i,j)=(2k+1,k+3), \\[1mm]
d_{2k+1-l} \>\> (1\le l\le k-1) &\>\>\mbox{ if }\>\>  j=k+l+2 \>\>\mbox{ and }\>\> i\le 2k+1-l, \>\>\mbox{ or } \\[1mm] 
 & \>\>\mbox{ \phantom{if} }\>\> j=k+l+3 \>\>\mbox{ and }\>\> i\ge 2k+1-l, \\[1mm]
d_{k+1}  &\>\>\mbox{ if }\>\> j=2k+2 \>\>\mbox{ and }\>\> i\le k+1.
\end{array}
\right.
\]
Then one can check that any triangle satisfies the triangle inequality in the same way as in the previous case. 
The proof is simpler since \eqref{la-mu} can be replaced by 
$|j+i-2k-3|\le|j-i|_{2k+2}-1$. 

\medskip
(3) Assume $n=4k$ $(k\ge2)$. 
Recall that a regular $4k$-gon has $4k$ edges of length $d_1, d_2, \dots, d_{2k-1}$ and $2k$ edges of length $d_{2k}$. 
Consider two circular spaces with $2k$ vertices $X_A'=\{A_0, A_1,\dots, A_{2k-1}\}$ and $X_B'=\{B_0, B_1, \dots, B_{2k-1}\}$ with 
\[
d(A_i,A_j)=d(B_i,B_j)=d_{|j-i|_{2k}}.
\]
\begin{figure}[htbp]
\begin{center}
\includegraphics[width=.46\linewidth]{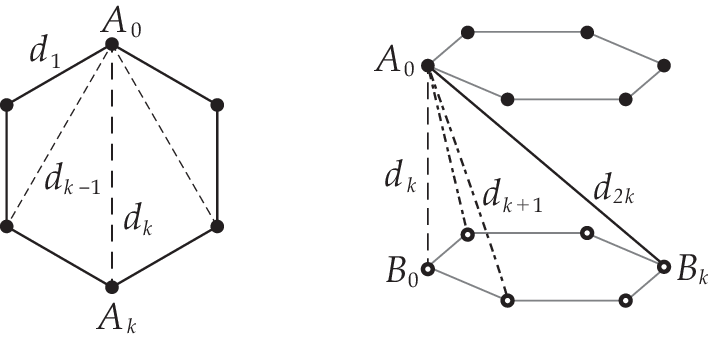}
\caption{$n/2$ is even}
\label{n=4k_bis}
\end{center}
\end{figure}
Then there are $4k$ edges of length $d_1, d_2, \dots, d_{k-1}$ and $2k$ edges of length $d_{k}$ so far. 
Assume 
\[
d(A_i,B_j)=d_{k+|j-i|_{2k}}
\]
(Figure \ref{n=4k_bis}). 
Then each vertex has two edges of length $d_1, \dots, d_{2k-1}$ and an edge of length $d_{2k}$. 
Put $X'=X_A'\cup X_B'$. 

Assuming that $X'$ is a metric space, it is not congruent to $X$ since $X'$ has two cycles consisting of $2k$ edges with length $d_1$, whereas $X$ does not. 
Now we have only to show the triangle inequalities. 

Lemma \ref{lem_small_m-gon} implies that if all the vertices of a triangle belong to $X_A'$ (or $X_B'$) then the triangle inequality is satisfied. 
Therefore we may assume, without loss of generality, that two vertices of the triangle belong to $X_A'$ and one to $X_B'$. Let the vertices be $A_0, A_j$ and $B_i$. We may assume by symmetry that $1\le j\le k$. 

(i) Suppose $0\le i\le j$. Then $d(A_0,B_i)=d_{k+i}, \, d(B_i,A_j)=d_{k+j-i}$, and $d(A_j,A_0)=d_j$. 
The shortest edge is $A_0A_j$. 
The triangle inequality is satisfied since 
\[
\begin{array}{l}
d_j+d_{k+i}>d_j+d_{k-i}\ge d_{k+j-i}, \\[1mm]
d_j+d_{k+j-i}\ge d_i+d_k\ge d_{k+i}.
\end{array}
\]

(ii) Suppose $j\le i\le j+k$. Then $d(A_0,A_j)=d_j$ and $d(A_j,B_i)=d_{k+i-j}$, and $d(B_i,A_0)=d_{k+|i|_{2k}}$. 
The shortest edge is $A_0A_j$. Remark that $k+|i|_{2k}=|k+i|_{4k}$. The triangle inequality is satisfied since 
\[
\begin{array}{l}
d_j+d_{k+i-j}\ge d_{\absn{k+i}}=d_{|k+i|_{4k}}, \\[1mm]
d_j+d_{|k+i|_{4k}}=d_{\absn{-j}}+d_{\absn{k+i}}\ge d_{\absn{k+i-j}}=d_{k+i-j}, 
\end{array}
\]
where the last inequality is the consequence of the circular triangle inequality \eqref{triangle_inequality_delta}. 

(iii) Suppose $j+k+1\le i\le 2k-1$. 
Then $d(A_0,A_j)=d_j, \, d(A_j,B_i)=d_{3k-i+j}$, and $d(B_i,A_0)=d_{3k-i}$, and the longest edge is $A_jB_i$. The triangle inequality is satisfied as 
$d_j+d_{3k-i}\ge d_{3k-i+j}$.

\medskip
(4) Assume $n=4k+2$ $(k\ge1)$. 
Recall that a circular space with $4k+2$ vertices has $4k+2$ edges of length $d_1, d_2, \dots, d_{2k}$ and $2k+1$ edges of length $d_{2k+1}$. 
Consider two circular spaces with $2k+1$ vertices $X_A'=\{A_0, A_1, \dots, A_{2k}\}$ and $X_B'=\{B_0, B_1, \dots, B_{2k}\}$ with 
\[
\begin{array}{c}
d(A_i,A_j)=d(B_i,B_j)=d_{|j-i|_{2k+1}}, \\[1mm]
d(A_i,B_j)=d_{2k+1-|j-i|_{2k+1}} 
\end{array}
\]
(Figure \ref{n=4k+2_bis}). 
\begin{figure}[htbp]
\begin{center}
\includegraphics[width=.43\linewidth]{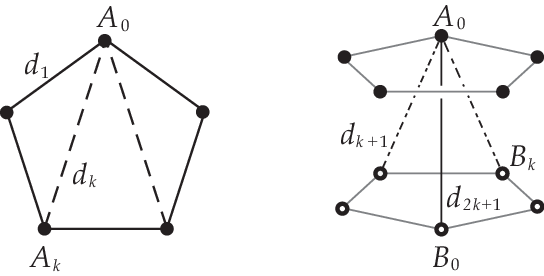}
\caption{$n/2$ is odd}
\label{n=4k+2_bis}
\end{center}
\end{figure}
Then each vertex has two edges of length $d_1, \dots, d_{2k}$ and an edge of length $d_{2k+1}$. 
The rest 
can be verified in the same way as in the previous case. 
\end{proof}

We note that if $n$ is even \jm{and} greater than or equal to $6$, the space $X'$ constructed above allows the isometric action of the dihedral group $D_{n}$ as well as $X$. 

There is \jm{an} alternative way for a regular $n$-gon $R_n$ when $n\ge9$. 
\begin{lemma}\label{lem_de_1_de_k-2}
Let $\delta_\bullet$ be given by \eqref{delta}. 
If $n=2k+1\ge 9$ and $1\le i\le k-2$ then $\delta_i+\delta_{k-1-i}\geq \delta_{k}$, 
and equality holds if and only if $n=9$ and $i=1$. 
\end{lemma}

\begin{proof}
First we show the inequality when $i=1$. 
If $n\geq 9$, i.e. $k\geq4$, we have 
\[
\frac13\le\frac{k-1}n=\frac{k-1}{2k+1}< \frac12, 
\]
which implies
\begin{equation*}
\begin{split}
\delta_{k}-\delta_{k-2}
&= \frac{1}{\sin(\pi/n)}\left(\sin\left(\frac{\pi k}n\right)- \sin\left(\frac{\pi(k-2)}n\right)\right)\\
&= \frac{2}{\sin(\pi/n)} \cos\left(\frac{\pi (k-1)}n\right) \sin\left(\frac\pi n\right)\\
&= 2 \cos\left(\frac{\pi (k-1)}n\right) \\
&\le 1 =  \delta_1,
\end{split}
\end{equation*}
with equality being satisfied if and only if $n=9$. 

If $i>1$ then the \jm{strict} concavity of $\theta\mapsto \sin\theta$ on $[0,\pi]$ implies
\begin{equation*}
\begin{split}
\delta_i+\delta_{k-1-i}
&= \frac{1}{\sin(\pi/n)} \left(\sin\left(\frac{\pi i}n\right)+ \sin\left(\frac{\pi(k-1-i)}n\right)\right)\\
&> \frac{1}{\sin(\pi/n)} \left(\sin\left(\frac{\pi }n\right)+ \sin\left(\frac{\pi(k-2)}n\right)\right)\\
&= \delta_1+\delta_{k-2} \\
&\ge \de_k.
\end{split}
\end{equation*}
\end{proof}

\begin{proposition}\label{isomer_Rn_odd}
\begin{enumerate}
\item Suppose $n=2k+1$ $(k\ge 4)$. Let $R_n'=\{P_0',\dots,P_{n-1}'\}$ be obtained from a regular $n$-gon by exchanging the lengths of the longest and second longest diagonals; 
\begin{equation}
d(P_i',P_j')=\left\{\begin{array}{ll}
\de_{|i-j|_n} & \quad \mbox{if } \>\>\> |i-j|_n\le k-2, \\
\de_k & \quad \mbox{if } \>\>\> |i-j|_n=k-1, \\
\de_{k-1} & \quad \mbox{if } \>\>\> |i-j|_n=k. 
\end{array}\right. 
\nonumber
\end{equation}
Then $R_n'$ is a metric space homometric to, but not congruent to $R_n$. 
\item Suppose $n=9$. Put $R_9''=\{A_0,A_1,A_2,B_0,B_1,B_2,C_0,C_1,C_2\}$ with \[
\begin{array}{ll}
d(A_i,A_j)=d(B_i,B_j)=d(C_i,C_j)=1 &\quad (i\ne j),\\[1mm]
d(A_i,B_i)=d(B_i,C_i)=d(C_i,A_i)=\de_2 &\quad (\forall i),\\[1mm]
d(A_i,B_{i+1})=d(B_i,C_{i+1})=d(C_i,A_{i+1})=\de_3 &\quad (\forall i),\\[1mm]
d(A_i,B_{i+2})=d(B_i,C_{i+2})=d(C_i,A_{i+2})=\de_4 &\quad (\forall i),
\end{array}\]
where subscripts are considered modulo $3$. 
Then $R_9''$ is a metric space homometric to, but not congruent to $R_9$. 
\end{enumerate}
\end{proposition}

\begin{proof}

(1) $R_n'$ is not congruent to $R_n$ since $R_n'$ has a triple of points with edge lengths $\de_1,\delta_{k-2},\delta_{k}$, whereas $R_n$ does not. 
It remains to verify that the triangle inequality is satisfied by any triangle, namely, we have to show $a'+b'\ge c'$ for any triangle $\triangle P_i'P_j'P_m'$ in $R_n'$ with edge lengths $a',b'$ and $c'$. 
Since the triangle inequality is satisfied for any triangle $\triangle P_iP_jP_m$ in $R_n$ with edge lengths $a,b$ and $c$, we have only to show $a'+b'\ge c'$ when $a'<a$ or $c'>c$. 

(i) $a'<a$ occurs only if $a'=\de_{k-1}$. In this case 
$a'+b'\ge\de_{k-1}+\de_1>\de_k\ge c'$. 

(ii) $c'>c$ occurs only if $c'=\de_k$. 
We may assume without loss of generality that $(P_i', P_j', P_m')=(P_i', P_0', P_{k-1}')$ 
$(1\le i\le n-1, \,i\ne k-1)$. 

(ii-1) When $1\le i\le k-2$, 
Lemma \ref{lem_de_1_de_k-2} implies $d(P_0',P_i')+d(P_i',P_{k-1}')=\delta_i+\delta_{k-1-i}\ge \de_k$.

(ii-2) When $i=k,k+1,k+2, 2k-2,2k-1,2k$, either $d(P_0',P_i')$ or $d(P_i',P_{k-1}')$ is equal to $\de_{k-1}$ or $\de_k$, which implies 
$d(P_0',P_i')+d(P_i',P_{k-1}') \ge \de_1+\de_{k-1}>\de_k.$

(ii-3) When  $i=k+3,\dots,2k-3$ ($k\geq6$), using the symmetry $P_m'\leftrightarrow P_{n-m}'$, if we put $j=n-i$ $(4\le j\le k-2)$ we have 
\[
d(P_0',P_i')+d(P_i',P_{k-1}')=d(P_0',P_j')+d(P_j',P_{k+2}'). 
\]
Since $4\le k+2-j\le k-2$ and $\de_m$ is increasing with respect to $m$, 
\[
{\rm RHS}>d(P_0',P_j')+d(P_j',P_{k-1}')\ge\de_k, 
\]
where the last inequality is given in the case (ii-1). 

\smallskip
(2) Note that $\de_1=1, \de_2\approx1.88, \de_3\approx2.53, \de_4=1+\de_2\approx 2.88$. Since $R_9''$ contains no triangles with side lengths $1,1,\de_3$ and $1,1,\de_4$, $R_9''$ is a metric space. 
Since from each vertex of $R_9''$ there emanate \jm{two edges of each of the} lengths $\de_1, \de_2, \de_3$, and $\de_4$, respectively, it follows that $R_9''$ and the regular nonagon $R_9$ are homometric. 
\end{proof}

\section{Identification of regular polygons by distance multisets}

Theorem \ref{thm_homometric} implies 

\begin{corollary}\label{thm_Riesz}
Neither an $n$-cycle graph nor the regular $n$-gon $(n\ge4)$ can be identified by the distance multiset. 
\end{corollary}

In what follows we only study the regular $n$-gons. 
The above corollary implies that in order to identify a regular polygon by its distance multiset, it is necessary to restrict the range of the metric spaces to be handled. 
One way is to restrict it to subspaces of Euclidean spaces, however this is not sufficient 
as \jm{the} non-congruent homometric spaces constructed in Theorem \ref{thm_homometric} can be realized in the Euclidean spaces as follows. 

When $n=4$ (cf. Figure \ref{four_pts} right) by 
\[(0,0,0), \>
(1,0,0), \>
(0,0,1), \>
\left(\frac12,\frac{\sqrt3}2,0\right), 
\]
when $n=5$ (cf. Figure \ref{five_pts_bis} right) by 
\[\begin{array}{l}
\displaystyle 
\left(\frac{\sqrt5+3}4,0,\frac{\sqrt5+1}4,0\right), 
\left(0,0,\frac{\sqrt5-1}4,\frac{\sqrt5+1}4\right), 
\left(0,0,\frac{\sqrt5-1}4,-\frac{\sqrt5+1}4\right), 
\\[4mm]  \displaystyle 
\left(0,\frac12,0,0\right), 
\left(0,-\frac12,0,0\right),
\end{array}
\]
when $n=6$ (cf. Figure \jm{\ref{n=4k+2_bis}}) by 
\begin{equation}\label{isomer_hexagon}
\begin{array}{l}
\displaystyle 
\left(\frac1{\sqrt3},0,0\right), 
\left(-\frac1{2\sqrt3},\frac12,0\right), 
\left(-\frac1{2\sqrt3},-\frac12,0\right), \\[4mm]
\displaystyle 
\left(-\frac1{\sqrt3},0,\frac{2\sqrt2}{\sqrt3}\right), 
\left(\frac1{2\sqrt3},\frac12,\frac{2\sqrt2}{\sqrt3}\right), 
\left(\frac1{2\sqrt3},-\frac12,\frac{2\sqrt2}{\sqrt3}\right).
\end{array}
\end{equation}
These examples show that in order to identify the regular polygons by the distance multiset among subsets of Euclidean spaces, the dimension of the Euclidean space must be $2$. 

\begin{remark}\label{rm_6-isomer_7_embed}\rm \begin{enumerate}
\item When $n=6$ there is another way to construct a non-congruent homometric space of the regular hexagon. Set $d(P_i,P_{i+1})=1$, where subscripts are considered modulo $6$, $d(P_0,P_3)=d(P_1,P_5)=d(P_2,P_4)=2$, and other edges having length $\sqrt3$. This space cannot be embedded isometrically in the Euclidean space. 
\item When $n=7$, Schoenberg's Theorem implies that the space constructed in Theorem \ref{thm_homometric} can be embedded isometrically in $\RR^6$. 
In fact, putting $H=(\de_{ij}-1/7)_{i,j}$, where $\de_{ij}$ is the Kronecker delta, and 
\setlength\arraycolsep{2.0pt}
\[
D^{(2)}=\left(\begin{array}{ccccccc}
0 & 1 & 1 & 1 & 1 & \de_2^{\,2} & \de_2^{\,2} \\
1 & 0 & 1 & 1 & \de_2^{\,2} & \de_2^{\,2} & \de_2^{\,2} \\
1 & 1 & 0 & 1 & \de_2^{\,2} & \de_3^{\,2} & \de_3^{\,2} \\
1 & 1 & 1 & 0 & \de_2^{\,2} & \de_3^{\,2} & \de_3^{\,2} \\
1 & \de_2^{\,2} & \de_2^{\,2} & \de_2^{\,2} & 0 & \de_3^{\,2} & \de_3^{\,2} \\
\de_2^{\,2} & \de_2^{\,2} &  \de_3^{\,2} & \de_3^{\,2} & \de_3^{\,2} & 0 & \de_3^{\,2} \\
\de_2^{\,2} & \de_2^{\,2} &  \de_3^{\,2} & \de_3^{\,2} & \de_3^{\,2} & \de_3^{\,2} & 0
\end{array}
\right), \quad
G=-\frac12 HD^{(2)}H, 
\]
where $\de_i=\sin(i \pi/7)/\sin(\pi/7)$, $G$ does not have negative eigenvalues and has rank $6$. 
\end{enumerate}
\end{remark}

Hereafter we assume that all the points are in the same plane $\RR^2$. 
We write $AB=d(A,B)$ for $A,B\in\RR^2$ when there is no fear of confusion.

\begin{lemma}\label{lemma_Fijk}
Suppose $n$ is a multiple of $3$ that is greater than or equal to $6$. 
Put $\theta=\pi/n$ and define $F_n(i,j,k)$ for $1\le i,j,k \le \lfloor n/2\rfloor$ by 
\begin{equation}\label{Fnijk}
\begin{array}{rcl}
F_n(i,j,k)&=& \sin^4\theta+
\sin^4i\theta+\sin^4j\theta+\sin^4k\theta \\[1mm]
&&-\sin^2i\theta\sin^2j\theta-\sin^2i\theta\sin^2k\theta-\sin^2j\theta\sin^2k\theta  \\[1mm]
&&-\sin^2\theta\left(\sin^2i\theta+\sin^2j\theta+\sin^2k\theta\right).
\end{array} 
\end{equation}
\begin{enumerate}
\item If there exists a point whose distances from the vertices of a unit equilateral triangle $\triangle$  
are $\de_i,\de_j$, and $\de_k$ $(1\le i,j,k\le\lfloor n/2\rfloor)$ then $i,j$ and $k$ satisfy the equation $F_n(i,j,k)=0$. 
\item The equation 
\begin{equation}\label{eq_Fnijk}
F_n(i,j,k)=0\qquad(1\le i\le j\le k\le \lfloor n/2\rfloor)
\end{equation}
has a solution $(n/3-1, n/3, n/3+1)$, and 
if $n=6$ it has an additional solution $(1,1,2)$, and if $n$ is even and $n\ge12$ then it has additional solutions $(n/6-1, n/6, n/6)$ and $(n/6, n/6, n/6+1)$. 
\item If $6\le n\le 30$ and $n\ne24$ then the equation \eqref{eq_Fnijk} does not have any more solutions, and if $n=24$ then it has only one additional solution $(8,10,11)$. 
\item The six points whose distances from the vertices of a unit equilateral triangle 
are $\de_{n/3-1}, \de_{n/3}$, and $\de_{n/3+1}$, which correspond to the solution $(n/3-1, n/3, n/3+1)$ to the equation \eqref{eq_Fnijk}, are located on the lines obtained by extending the edges of the triangle so that the distance from the closest vertex is equal to $\de_{n/3-1}$ as illustrated in {\rm Figure \ref{ijk_position}}. 
\begin{figure}[htbp]
\begin{center}
\includegraphics[width=7cm]{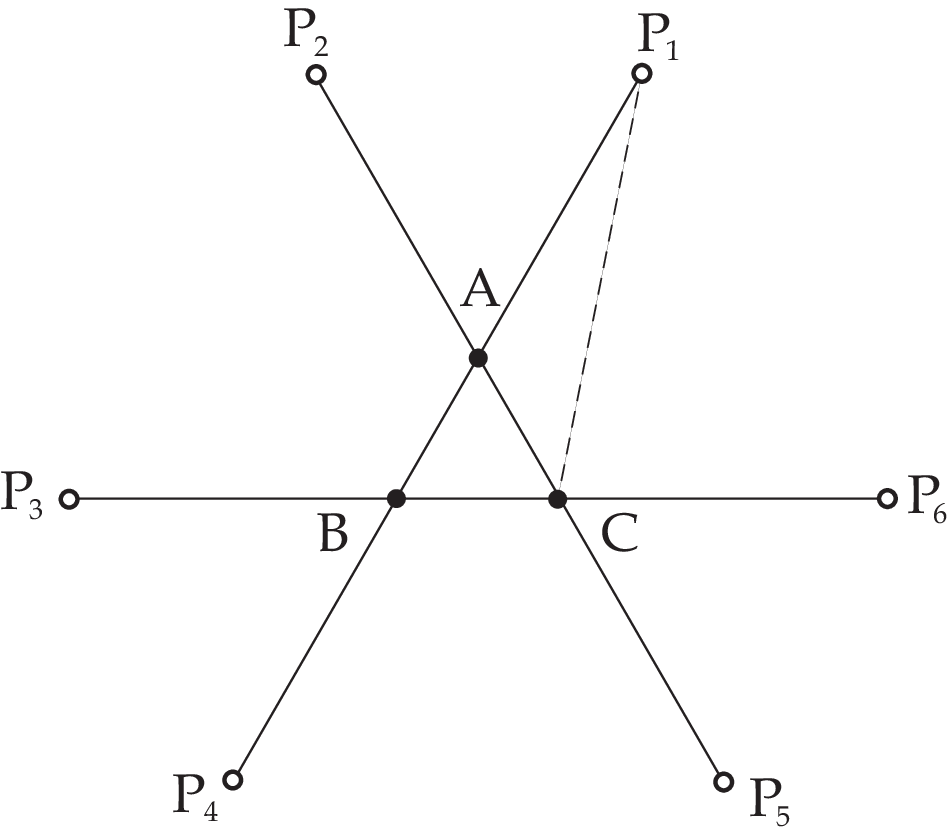}
\caption{$P_1A=\de_{n/3-1}, P_1B=\de_{n/3+1}$ and $P_1C=\de_{n/3}$. 
$P_2A=P_3B=P_4B=P_5C=P_6C=\de_{n/3-1}$.}
\label{ijk_position}
\end{center}
\end{figure}
\end{enumerate}
\end{lemma}

\begin{proof}
(1) Let $B=(0,0), C=(1,0)$ and $A=(1/2, \sqrt3/2)$. The point $P$ with $PB=\de_j$ and $PC=\de_k$ is given by 
\[
P=\left(
\frac{\de_j^2-\de_k^2+1}2, \pm\frac{\sqrt{-\de_j^4-\de_k^4+2\de_j^2\de_k^2+2\de_j^2+2\de_k^2-1}}2
\right). 
\]
Then $PA^2-\de_i^2=0$ gives $F_n(i,j,k)=0$. 

(2) 
Noting that $(n/3)\theta=\pi/3$, we obtain
\[
F_n\left(\frac{n}3-1,\frac{n}3,\frac{n}3+1\right)=
\frac9{16}\left(\sin^4\theta+\cos^4\theta+2\sin^2\theta\cos^2\theta-2\sin^2\theta-2\cos^2\theta+1\right)=0.
\]
Similarly, as $(n/6)\theta=\pi/6$, we obtain
\[
F_n\left(\frac{n}6,\frac{n}6,\frac{n}6\pm1\right)=\frac1{16}\left(
1+13\sin^4\theta+\cos^4\theta+14\sin^2\theta\cos^2\theta-14\sin^2\theta-2\cos^2\theta
\right)=0.
\]
When $n=24$, $F_{24}(8,10,11)=0$ follows from 
\[
\begin{array}{rlrl}
\sin\theta&=\displaystyle \sin\frac\pi{24}=\sqrt{\frac{4-\sqrt2-\sqrt6}8}, \>&\> 
\sin(8\theta)&=\displaystyle \sin\frac\pi3=\frac{\sqrt3}2, \\
\sin(10\,\theta)&=\displaystyle \sin\left(\frac\pi3+\frac\pi{12}\right)=\frac{\sqrt2+\sqrt6}{4},\>&\>
\sin(11\,\theta)&=\displaystyle \sin\left(\frac\pi3+\frac\pi{8}\right)=\sqrt{\frac{4+\sqrt2+\sqrt6}8}.
\end{array}
\]

(3) can be shown by computing all the values of $F_n(i,j,k)$ ($6\le n\le30, 1\le i\le j\le k\le \lfloor n/2\rfloor$) with the help of a computer.\footnote{The second author used Maple for this computation.} 

(4) follows from the fact that $1+\de_{n/3-1}=\de_{n/3+1}$. 
\end{proof}

\begin{theorem}\label{thm_beta}
The regular $n$-gon can be identified by the distance multiset in the set of tuples of $n$ points in the plane. 
\end{theorem}

\begin{proof}
Part I. Suppose $n\le6$. 

(1) Case $n=4$. If there is a vertex $P$ with two sides of length $\sqrt2$, the other three points form an equilateral triangle so that the distances from the point $P$ to the vertices are $\sqrt2, \sqrt2$ and $1$, which cannot be realized in a plane. If two sides of length $\sqrt2$ do not have a common vertex then the configuration is isometric to the square.  

\smallskip
(2) Case $n=5$. 
Assume there is a vertex that has three edges of length $1$. 
The multiset formed by the lengths of the edges connecting the endpoints of these three edges is either $1,1,\de_2$ or $1,\de_2,\de_2$ or $\de_2,\de_2,\de_2$, which cannot happen in a plane. 
It follows that each vertex has exactly two edges of length $1$. 
Since there is no multi-edge, length $1$ edges must form a pentagon. 
Connecting the remaining parts with edges of length $\de_2$, we get a regular pentagon. 

\smallskip
(3) Case $n=6$. Note that $\de_2=\sqrt3$ and $\de_3=2$. 
Assume there is a vertex that has three edges of length $1$. 
The multiset formed by the lengths of the edges connecting the endpoints of these three edges is either $\sqrt3, \sqrt3, \sqrt3$ or $1,1,\sqrt3$ or $1,\sqrt3,2$. Let these three cases be (i), (ii) and (iii) respectively 
(Figure \ref{n=6_123}). 
\begin{figure}[htbp]
\begin{center}
\includegraphics[width=9.6cm]{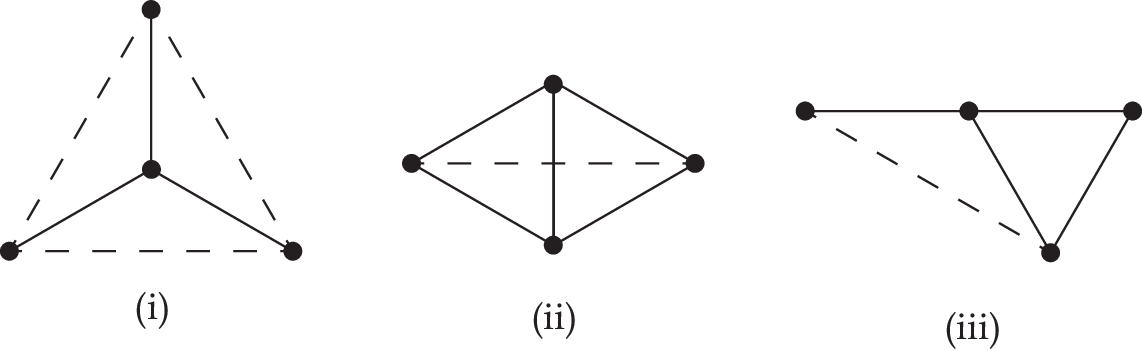}
\caption{The solid lines represent edges of length $1$ and the dashed $\sqrt3$.}
\label{n=6_123}
\end{center}
\end{figure}

(3-i) Let the four vertices be $A,B,C$ and $D$ with $D$ being the center. 
The remaining two vertices must satisfy the condition that their distance from points $A,B,C$ or $D$ is either $1,\sqrt3$, or $2$. 
But $\{X\,|\,1\le XA,XB,XC\le2\}$ consists of three points, $P,Q$ and $R$ as illustrated in Figure \ref{n=6_1}, 
\begin{figure}[htbp]
\begin{center}
\includegraphics[width=5cm]{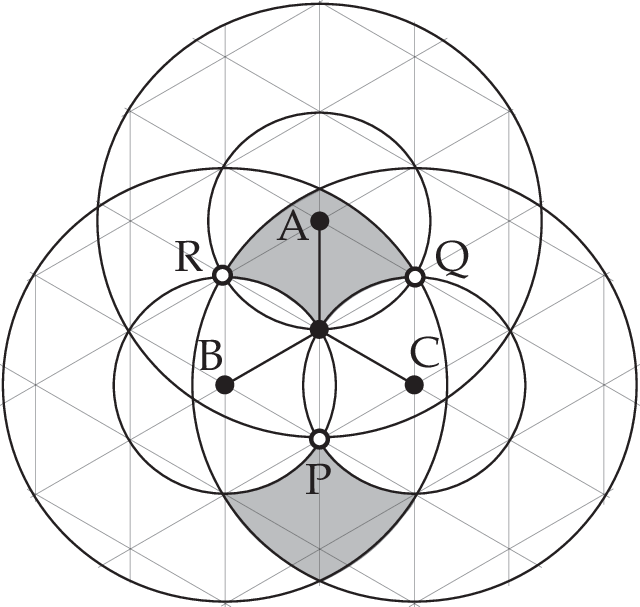}
\caption{Circles with centers $A, B, C$ and radii $1$ and $2$. The gray part is $\{X\,|\,1\le XB,XC\le2\}$}
\label{n=6_1}
\end{center}
\end{figure}
and any two points from $P,Q$ and $R$ would create nine sides of length $1$ in all. 
Therefore the case (i) cannot happen.

(3-ii) Let the four vertices be $A,B,C$ and $D$ as illustrated in Figure \ref{n=6_2bis}. 
Lemma \ref{lemma_Fijk} shows that the set $\{X\,|\,XA,XB,XC\in\{1,\sqrt3,2\}\}$ is given as in Figure \ref{equi_triangle_9pts}, which implies that $\{X\,|\,XA,XB,XC,XD\in\{1,\sqrt3,2\}\}$ consists of six points, $P, \dots, R'$ as illustrated in Figure \ref{n=6_2bis}. 
\begin{figure}[htbp]
\begin{center}
\begin{minipage}{.45\linewidth}
\begin{center}
\includegraphics[width=3.0cm]{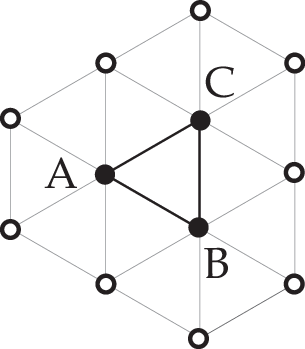}
\caption{$\{X\,|\,XA,XB,XC\in\{1,\sqrt3,2\}\}$}
\label{equi_triangle_9pts}
\end{center}
\end{minipage}
\hskip 0.1cm
\begin{minipage}{.5\linewidth}
\begin{center}
\includegraphics[width=2.6cm]{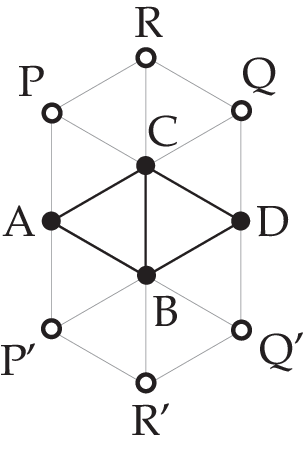}
\caption{$\{X\,|\,XA,XB,XC,XD\in\{1,\sqrt3,2\}\}$}
\label{n=6_2bis}
\end{center}
\end{minipage}
\end{center}
\end{figure}

Since any two points from these six points 
would create more than six sides of length $1$ in all, the case (ii) cannot happen.

(3-iii) Let the four vertices be $A,B,C$ and $D$ as in Figure \ref{n=6_3}. 
Then Lemma \ref{lemma_Fijk} implies that $\{X\,|\, XA,XB,XC,XD\in\{1,\sqrt3,2\}\}=\{P,Q,R,S\}$, where $P,Q,R$ and $S$ are given as in Figure \ref{n=6_3}. 
\begin{figure}[htbp]
\begin{center}
\includegraphics[width=2.4cm]{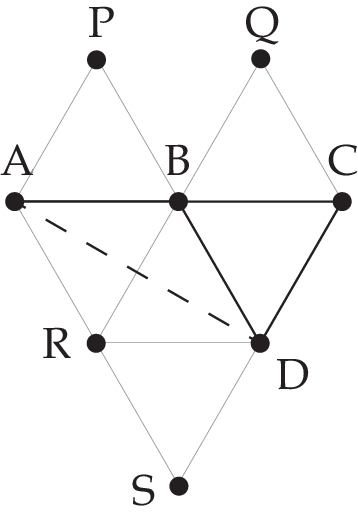}
\caption{}
\label{n=6_3}
\end{center}
\end{figure}

But any two points from $P,Q,R$ and $S$ would create more than six edges of length $1$ in all. Therefore the case (iii) cannot happen. 

\jm{This} completes the proof that there is no vertex that has three edges of length $1$. Therefore, the configuration of edges of length $1$ is either a (topological) hexagon or a disjoint union of two triangles. 
But since Figure \ref{equi_triangle_9pts} shows that it cannot be a disjoint union of two triangles, it should be a (topological) hexagon. 
Let it be $P_0\dots P_5$. 
Assume that $P_{i-1}P_{i+1}=2$ for some $i$. 
Then, 
since $\{X\,|\,XP_{i-1}, XP_{i}, XP_{i+1}\in\{1,\sqrt3,2\}\}$ consists of six points as illustrated in Figure \ref{n=6_three_pts}, the only possible configuration is given by Figure \ref{n=6_big_triangle}, 
\begin{figure}[htbp]
\begin{center}
\begin{minipage}{.6\linewidth}
\begin{center}
\includegraphics[width=5.4cm]{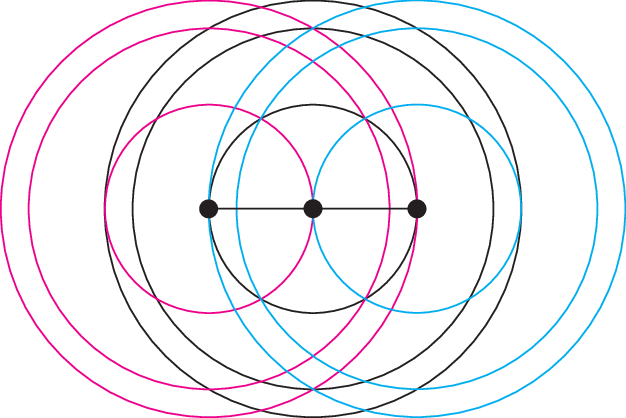}
\caption{Circles with radii $1,\sqrt3,2$ with centers $P_{i-1},P_i,P_{i+1}$}
\label{n=6_three_pts}
\end{center}
\end{minipage}
\hskip 0.4cm
\begin{minipage}{.29\linewidth}
\begin{center}
\includegraphics[width=1.8cm]{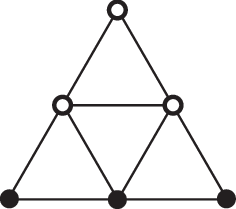}
\caption{}
\label{n=6_big_triangle}
\end{center}
\end{minipage}
\end{center}
\end{figure}
where there are nine edges of length $1$, which is not allowed. 
Therefore, each $P_{i-1}P_{i+1}$ is equal to $\sqrt3$. 
To close, all edges must turn in the same direction, resulting in a regular hexagon, which completes the proof of the case $n=6$. 

\medskip
Part II. Suppose $n\ge7$. 

First remark that $\de_3=\sin(3\pi/n)/\sin(\pi/n)>2$ when $n>6$. 
Consider the shortest edge graph $G$, i.e. the graph with the edges with length $1$. 
Assume $G$ has a trivalent vertex. 
The multiset formed by the distances between the endpoints of three edges emanating from the trivalent vertex is either $[1,1,\de_2]$ or $[\de_2, \de_2, \de_2]$ or $[1,\de_2,\de_2]$. 
\jm{The f}irst two cases are impossible since $n\ne6$ and the last case can happen only if $n=12$. 
Remark that in the latter case there appears an equilateral triangle with side length $1$. 
\begin{figure}[htbp]
\begin{center}
\includegraphics[width=2.0cm]{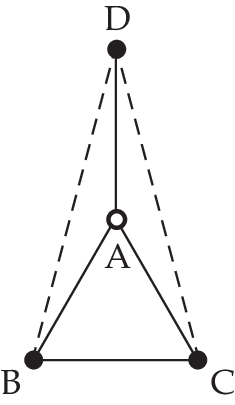}
\caption{$n=12$, $DB=DC=\de_2$}
\label{n=12_trivalent}
\end{center}
\end{figure}

If there is no vertex that has three edges of length $1$, in particular when $n\ne12$, 
the shortest edge graph $G$ is a union of cycles. 
Since the distance between the endpoints of adjacent edges is either $1$ or $\de_2$, the cycle is either a\jm{n} equilateral triangle or a regular $n$-gon. 

\smallskip
(4) Case when 
$n$ is not a multiple of $3$. 
Since $n\ne12$, the shortest edge graph $G$ 
is a union of cycles. Since it cannot be a union of triangles, it should be a regular $n$-gon. 

\smallskip
(5) Case when $n$ is a multiple of $3$. 
We have only to show that the shortest edge graph $G$ does not contain an equilateral triangle. 
Lemma \ref{lemma_Fijk} implies that if there is an equilateral triangle, 
all the other vertices correspond to the solutions of \eqref{eq_Fnijk}. 
The solutions $(n/6-1, n/6, n/6)$ and $(n/6, n/6, n/6+1)$ each produce three points, and the solutions $(n/3-1, n/3, n/3+1)$ and $(8,10,11)$ each produce six points. 

\smallskip
(5-i) Case $n=9$. 
The equation \eqref{eq_Fnijk} has a unique solution $(2,3,4)$ which corresponds to six points as illustrated in Figure \ref{ijk_position}. 
Since $P_1P_2=\de_2>1$ and $P_2P_3=1+\de_2>2$, there are no more edges with length $1$ except for the three edges of $\triangle ABC$. 

\smallskip
(5-ii) Case $n=12$. 
All the points whose distances from the vertices of an equilateral triangle $\triangle ABC$ belong to $\{\de_1,\dots,\de_6\}$ are illustrated in Figure \ref{n=12_last}. 
\begin{figure}[htbp]
\begin{center}
\includegraphics[width=9cm]{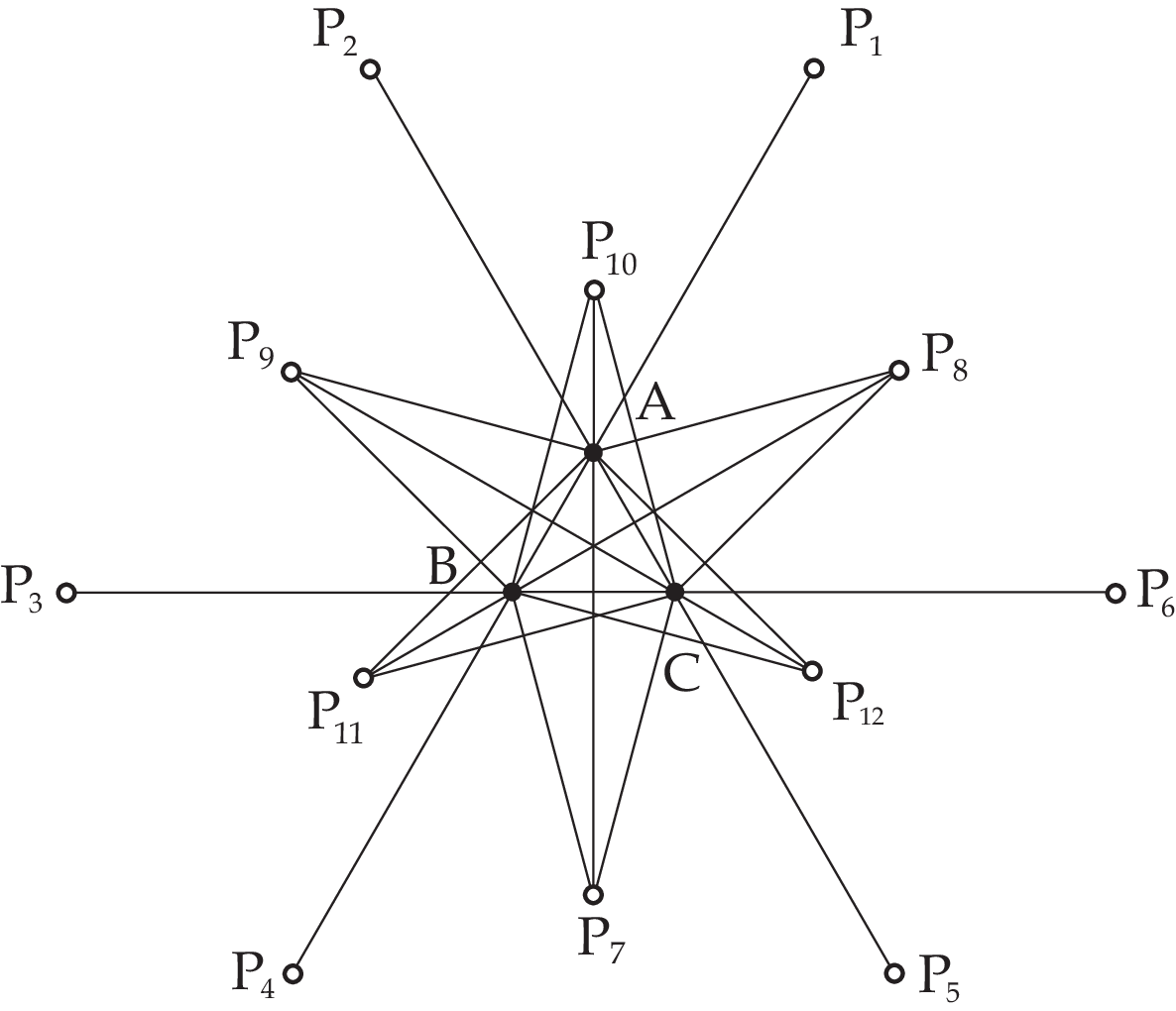}
\caption{}
\label{n=12_last}
\end{center}
\end{figure}
$P_1,\dots,P_6$ correspond to the solution $(3,4,5)$ to the equation \eqref{eq_Fnijk}, i.e., $P_1A=\de_3, P_1C=\de_4$ and $P_1B=\de_5$ {\it etc.}, $P_7,P_8$ and $P_9$ to $(2,2,3)$, i.e., $P_7B=P_7C=\de_2$ and $P_7A=\de_3$ {\it etc.}, and $P_{10}, P_{11}$ and $P_{12}$ to $(1,2,2)$, i.e., $P_{10}B=P_{10}C=\de_2$ and $P_{10}A=\de_1$ {\it etc.}
Among these fifteen points, $A,B,C,P_1,\dots,P_{12}$, there are only six pairs of points with distance $1$, which is not allowed. 

\smallskip
(5-iii) Case $n\ge15$. 
Assume that the shortest edge graph $G$ consists of $n/3$ triangles. 
Since all edges of length $1$ are used up, the distance between vertices in different triangles is bounded below by $\de_2$. Consequently, there must exist two vertices from distinct triangles at a distance $\de_2$. Let $A$ be such a vertex belonging to a triangle $ABC$, and let $D$ be the corresponding vertex of another triangle. \jm{T}he triplet of distances from $D$ to $A, B$, and $C$ must be either $\delta_2,\delta_2,\delta_3$ (Figure \ref{---triangle}) or $\delta_2,\delta_3,\delta_3$ (Figure \ref{triangle--}). However, the former is possible only when $n=12$, which contradicts our assumption $n\ge15$, and the latter occurs only when $n=18$. 
The solutions to \eqref{eq_Fnijk} correspond to twelve points when $n=18$. 
Adding these twelve points to the three vertices of the triangle does not give a total of $18$ points, which yields a contradiction.

\begin{figure}[htbp]
  \begin{minipage}[b]{0.45\textwidth}
    \centering
    \begin{tikzpicture}[scale=0.84]
    \def\L{2} 
    \coordinate (Top)   at (0, {\L/2});               
    \coordinate (Bot)   at (0, {-\L/2});              
    \coordinate (Right) at ({\L*sqrt(3)/2}, 0);       
    
    \def\deltaThree{3.5} 
    \coordinate (Left)  at (-\deltaThree, 0);         

    \draw (Top) -- (Bot) -- (Right) -- cycle; 
    \draw (Left) -- (Right);                 
    \draw (Left) -- (Top);                 
    \draw (Left) -- (Bot);                 
    
    \fill (Top)   circle (2.5pt);
    \fill (Bot)   circle (2.5pt);
    \fill (Right) circle (2.5pt);
    \fill (Left)  circle (2.5pt);

    \path (Left) -- (Top) node[midway, above] {$\delta_2$};
    \path (Left) -- (0,0) node[midway, below] {$\qquad\qquad\delta_3$};
    \path (Bot) -- (Right) node[midway, below right] {$1$};

  \end{tikzpicture}
  \caption{$\delta_2,\delta_2,\delta_3$}
\label{---triangle}
  \end{minipage}
\hspace{0.1cm}
  \hfill
  \begin{minipage}[b]{0.45\textwidth}
    \centering
  \begin{tikzpicture}[scale=0.84]
    
    \def\L{2} 
    
    \coordinate (B) at (0, 0);                 
    \coordinate (A) at (0, \L);                
    \coordinate (C) at ({\L*sqrt(3)/2}, {\L/2}); 
    
    \def\deltaTwo{3.5} 
    \coordinate (D) at ({\L*sqrt(3)/2 + \deltaTwo}, {\L/2});

    \draw (A) -- (B) -- (C) -- cycle; 
    \draw (C) -- (D);                 
    \draw (A) -- (D);                 
    \draw (B) -- (D);

    \fill (A)   circle (2.5pt);
    \fill (B)   circle (2.5pt);
    \fill (C) circle (2.5pt);
    \fill (D)  circle (2.5pt);

    \path (A) -- (B) node[midway, left] {$1$};
    \path (C) -- (D) node[midway, below left] {$\delta_2 \qquad{}$};
    \path (A) -- (D) node[midway, above right] {$\delta_3$};

  \end{tikzpicture}
  \caption{$\delta_2,\delta_3,\delta_3$}
\label{triangle--}
  \end{minipage}
\end{figure}

\end{proof}

We close the section with the explanation of the Riesz energy and its discrete version. 

\begin{remark}\label{rem_Riesz_mag}\rm 
The {\em discrete Riesz energy function} of a finite metric space $X=\{P_1,\dots,P_n\}$ is given by 
\[
B_X(z)=\sum_{i\ne j}d(P_i,P_j)^z \qquad (z\in \CC).
\]
It is equivalent to the distance multiset of $X$, $\mathcal{D}=[d_1,\dots,d_N]$, where $N={n\choose2}$, as follows. 
$B_X(z)$ is given by $\mathcal{D}$ by $B_X(z)=\sum_{1\le i\le N}{d_i}^z$. 
On the contrary, from $B_X(0)$ we obtain the number of points $n$, and from $B_X(1), \dots, B_X(N)$, where $N={n\choose2}$, we obtain all the elementary symmetric polynomials in $d_1,\dots,d_N$, from which we can obtain $\mathcal{D}$. 

The discrete Riesz energy function is a discrete version of the {\em Riesz energy function} of a compact manifold $X$ that is given as follows. 
Consider a map 
\[
\CC\ni z\mapsto \iint_{X\times X} {d(x,y)}^z\, \d\mu_x \d\mu_y \in\CC,
\]
where $\d\mu$ stands for a Lebesgue measure on $X$. 
It is a well-defined holomorphic function of $z$ on $\R z>-\dim X$. 
Extend the domain by analytic continuation to obtain a meromorphic function with only simple poles. We call it the {\em Riesz energy function} or the {\em Brylinski beta function} and denote it by $B_X(z)$ (\cite{Bry,FV,OS2}). 
\end{remark}

\section{Identification by magnitude}

\subsection{Definition of magnitude and magnitude homology}

The magnitude is a numerical invariant for a metric space introduced by Leinster \cite{L13}. 
Let $(X,d)$ be a finite metric space, $X=\{P_1,\dots, P_n\}$. 
The matrix $Z_X=\left(\exp(-d(P_i,P_j)\jm{)}\right)_{i,j}$ is called the {\em similarity matrix} of $X$. 
A column vector $w=\!{}^t(w_1,\dots,w_n)$ is \jm{a} {\em weighting} of $X$ if $Z_Xw=\vect 1$, where $\vect{1}=\!{}^t(1,\dots,1)$. When there is a weighting the {\em magnitude} of $X$ is given by $|X|=w_1+\dots+w_n$. 
When $Z_X$ is invertible the magnitude $|X|$ is given by the sum of all the entries of ${Z_X}^{-1}$; $|X|={}^t\vect{1}{Z_X}^{-1}\vect{1}$. 

Let $tX$ $(t>0)$ denote the metric space $(X,td)$. 
The {\em magnitude function} is the partially defined function $M_X(t)=|tX|$ defined for such $t$ that $tX$ has magnitude (\cite{L13}). 
For any finite metric space $X$ the similarity matrix of $tX$ is invertible for all but finitely many $t>0$ (\cite{L13} Proposition 2.2.6 i). 
By putting $q=e^{-t}$ the similarity matrix of $tX$ is expressed by $z_X(q)=\left(q^{d_{ij}}\right)_{i,j}$. 
The determinant, considered as a {\sl generalized polynomial} which is a polynomial that allow\jm{s} positive non-integer exponents, is invertible in the field of generalized rational functions since it has the constant term $1$. 
The sum of all the entries of $z_X(q)^{-1}$ is called the {\em formal magnitude} (\cite{L17}), and denoted by $m_X(q)$. 

Magnitude homology $MH^l_\ast(X)$ is defined for graphs by Hepworth and Willerton \cite{HW} and for enriched categories by Leinster and Shulman \cite{LS}. 
We follow Hepworth and Willerton. 
For $k,l\ge0$ let $\mathcal{P}_k^l(X)$ be the set of $k$-paths of $X$ with length $l$;
\begin{equation}\label{Pkl}
\mathcal{P}_k^l(X)=\left\{\gamma=(x_0,x_1,\dots,x_k)\in X^{k+1}\,|\,
\jm{x_{i-1}\ne x_i \>\>(1\le i\le k)}, 
\>\> \ell(\gamma)=l
\right\},
\end{equation}
\jm{where $\ell(\gamma)=\sum_{i=1}^k d(x_{i-1},x_i)$. }
The magnitude chain complex is defined by $MC_k^l(X)=\ZZ \mathcal{P}_k^l(X)$ for $k,l\ge0$, 
and the boundary operator (differential) $\partial\colon MC_k^l(X)\to MC_{k-1}^l(X)$ is defined by 
\[
\partial(x_0,\dots,x_k)=\sum_{i=1}^{k-1}(-1)^i\partial_i(x_0,\dots,x_k),
\]
where
\[
\partial_i(x_0,\dots,x_k)=\left\{
\begin{array}{ll}
(x_0,\dots,\hat x_i,\dots,x_k)  & \quad \mbox{if} \>\>\> \ell(x_0,\dots,\hat x_i,\dots,x_k)=\ell(x_0,\dots,x_k), \\
0 & \quad \mbox{otherwise}.
\end{array}
\right.
\]
The {\em magnitude homology} is defined by $MH_k^l(X)=H_k(MC_\ast^l(X))$. 
Then formal magnitude can be considered as the Euler characteristic of the magnitude homology as 
\[
m_X(q)=\sum_{k,l\ge0}(-1)^k\,{\rm rank}\,(MH^l_k(X))\cdot q^l
\]
(\cite{HW} Theorem 2.8). 

\subsection{Definition of isomers}

Recall that a metric space $X$ is \jm{said} to be {\em homogeneous} if for any $x,y\in X$ there is a self-isometry of $X$ that maps $x$ to $y$. 

\begin{definition} \rm 
\begin{enumerate}
\item Let $X=\{P_1,\dots,P_{n}\}$ be a metric space. 
If there is a multiset $[d_1, \dots, d_{n-1}]$, where $d_i>0$ for any $i$, such that the multiset $[d(P_i,P_1),\dots,d(P_i,P_n)]$ is equal to $[0,d_1, \dots, d_{n-1}]$ for any $i$ ($1\le i\le n$), then we say $X$ is {\em quasi-homogeneous of type} $[d_1, \dots, d_{n-1}]$. 
\jm{When $n$ is odd, we assume that the multiset is of the form $[d_1,d_1,\dots,d_{(n-1)/2},d_{(n-1)/2}]$; that is, every value appears an even number of times in the multiset.}

\item Two finite metric spaces are called {\em isomers} if they are quasi-homogeneous of the same type but not congruent\footnote{We use this terminology, by analogy with chemistry, where objects share the same invariants but differ in structure.}.
\end{enumerate}
\end{definition}

\jm{We remark that when $n$ is even, the values of $d_1,\dots,d_{n-1}$ can be all distinct, 
because complete graph $K_n$ admits a proper edge-coloring with $n-1$ colors (cf. \cite{H} Ch. 9). }

A circular space is homogeneous, and a homogeneous space is quasi-homogeneous. 

\begin{proposition}\label{prop_isomer_eg}
The $n$-point circular space $X$ for even $n\ge6$ and its homometric space $X'$ constructed in Theorem \ref{thm_homometric}, the pairs $R_n$, $R_n'$ for odd $n\ge9$ in Proposition \ref{isomer_Rn_odd} (1), $R_9$, \jm{$R_9''$} in Proposition \ref{isomer_Rn_odd} (2), and the regular hexagon and its homometric space in Remark \ref{rm_6-isomer_7_embed} (1) provide examples of isomers. 
\end{proposition}

\begin{proposition}\label{prop_isomer_mag}
Isomers have the same distance multiset and formal magnitude. 
\end{proposition}

\begin{proof}
Suppose $X$ and $X'$ are quasi-homogeneous of type $[d_1, \dots, d_{n-1}]$. 
\jm{We assume that \[[d_1, \dots, d_{n-1}]=[d_1,d_1,\dots,d_{(n-1)/2},d_{(n-1)/2}]\] when $n$ is odd.} 
Put $d_0=0$. 

The first statement \jm{holds} since \jm{the distance multisets of both spaces are given by 
\[
[d(P_i,p_j)]_{i<j}=\left\{\begin{array}{ll}
\displaystyle \frac{n}2[d_1, \dots, d_{n-1}] & \quad \mbox{if $n$ is even},\\[4mm]
n[d_1, d_2,\dots, d_{(n-3)/2}, d_{(n-1)/2}] & \quad \mbox{if $n$ is odd.} 
\end{array}
\right.
\]
}

As for the formal magnitude, since $w=\!{}^t(w_0, \dots, w_0)$, where $w_0=\left(\sum_{j=0}^{n-1}q^{d_j}\right)^{-1}$, is the weighting for $X$ and $X'$, we have  
\begin{equation}\label{Speyer}
m_{X}(q)=m_{X'}(q)=\frac{n}{\sum_{j=0}^{n-1}q^{d_j}}.
\end{equation}
We remark that this equality was shown by Speyer (\cite{L13} Proposition 2.1.5) in the homogeneous case. 
\end{proof}

We give \jm{a} non-existence statement of isomers of cycle graphs. 

\begin{lemma}\label{lem_isometric}
Let $(X,d)$ and $(X,d')$ be finite metric spaces with the same base space $X$ that are both quasi-homogeneous of the same type. 
If $d(x,y)\le d'(x,y)$ for any $x$ and $y$ in $X$ then $(X,d)$ and $(X,d')$ are isometric. 
\end{lemma}

\begin{proposition}\label{non-exists_isomer_cycle}
The $n$-cycle graph does not have an isomer if $n$ is an odd number greater than or equal to $7$. 
\end{proposition}

\begin{proof}
Suppose $(X,d)$ is quasi-homogeneous of type $[1,1,2,2,\dots, k,k]$. 
Since for each point of $X$ there are two points with distance $1$, any point of $X$ belongs to a cycle consisting of three or more edges of length $1$, which we simply call a {\sl cycle} in this proof. 
Note that there are $n$ pairs of points with distance $k$. 

Case 1. Suppose there is only one such cycle. 
We may assume that $X=\ZZ_n$ and $d(x,y)=1$ if $|x-y|_n=1$. 
Then the triangle inequality implies that $d(x,y)\le|x-y|_n$ for any $x,y\in\ZZ_n$, which implies that $(X,d)$ is isometric to $C_n$ by Lemma \ref{lem_isometric}. 

Case 2. Suppose there are exactly two such cycles. 
Let the number of vertices be $m_1$ and $m_2$ $(m_1<m_2)$. 
Since $m_1\ge3$ there holds $m_1<m_2\le 2k-2$, which implies that in each cycle there is no pair of points with distance $k$. 
On the other hand, since each point has exactly two points with distance $k$ and since $m_1\le k$ because $m_1+m_2=2k+1$, there are at most $2m_1$ pairs of points with distance $k$, which is a contradiction as $2m_1<n$. 

Case 3. Suppose there are $p$ such cycles $L_1, \dots, L_p$, where $p\ge3$. 
Let $m_i$ be the number of vertices of $L_i$. 
Remark that there are no cycles that have a pair of points with distance $k$ since each $m_i$ is smaller than $2k$. 

(i) Assume there \jm{is} no triple of points $x,y,z$ such that $d(x,y)=d(y,z)=k$ and that $x,y$ and $z$ belong to mutually distinct cycles. 
Let $a_{ij}$ $(1\le i,j\le p)$ be the number of points $x$ in $L_i$ such that there is a point $y$ in $L_j$ with $d(x,y)=k$. 
Then $a_{ii}=0$, $a_{ji}=a_{ij}$ and $m_i=\sum_{j\ne i}a_{ij}$ for any $i$ $(1\le i\le p)$. 
The number of pairs of points with distance $k$ is given by $2\sum_{i<j}a_{ij}$, which is equal to $\sum_{i=1}^pm_i=n$, which contradicts \jm{the fact that $n$ is odd}. 

(ii) Assume there is a triple of points $x,y,z$ such that $d(x,y)=d(y,z)=k$ and that $x,y$ and $z$ belong to mutually distinct cycles, say $L_1, L_2$ and $L_3$. 
Let $x^-,x^+$ (or $z^-,z^+$) be points with distance $1$ from $x$ (or respectively, $z$). 
Then $x^-,x,x^+,z^-,z,z^+$ are mutually distinct since $x^-,x,x^+\in L_1$ and $z^-,z,z^+\in L_3$. 
The triangle inequality implies that $d(y,x^\pm), d(y,z^\pm)\ge k-1$. 
It follows that there are at least six points with distance not smaller than $k-1$ from $y$, which contradicts the assumption that $X$ is of type $[1,1,\dots, k-1,k-1,k,k]$. 
\end{proof}

\subsection{Identification by formal magnitude}

%
Speyer's formula \eqref{Speyer} shows that the formal magnitude of a circular space of type $(d_1,\dots, d_{\lfloor n/2 \rfloor})$ is given by  
\begin{equation}\label{fmrp}
m_{X}(q)=\left\{
\begin{array}{ll}
\displaystyle \frac{n}{1+2q^{d_1}+2q^{d_2}+\dots+2q^{d_{(n-1)/2}}} &\quad \mbox{if $n$ is odd,}\\[4mm]
\displaystyle \frac{n}{1+2q^{d_1}+2q^{d_2}+\dots+2q^{d_{n/2-1}}+q^{d_{n/2}}} &\quad \mbox{if $n$ is even}.
\end{array}
\right.
\end{equation}

\begin{theorem}\label{thm_chara_magnitude}
Assume $n\ge4$. 
\begin{enumerate}
\item The regular $n$-gon can be identified by the formal magnitude if and only if $n=4,5$ or $7$. 
\item If $n=4$ or $5$, the $n$-cycle graph can be identified by the formal magnitude. 
\item If $n$ is even and greater than or equal to $6$, then the $n$-cycle graph cannot be identified by the formal magnitude. 
\end{enumerate}
\end{theorem}

\begin{proof}
(1) 
``Only if'' part 
is the consequence of Propositions \ref{prop_isomer_eg} and \ref{prop_isomer_mag}.

\smallskip
``If'' part. 
The case 
$n=4$ is the consequence of Theorem 2.4 of \cite{O24} since the length of the diagonal of a square is less than twice the length of the sides. 

\smallskip

Figure \ref{five_pts} illustrates all the possible configurations for graphs with five vertices and five edges. 
Among these configurations, the number of {\em unoriented} open $2$-paths is greater than five except for the pentagon. 
\begin{figure}[htbp]
\begin{center}
\includegraphics[width=14cm]{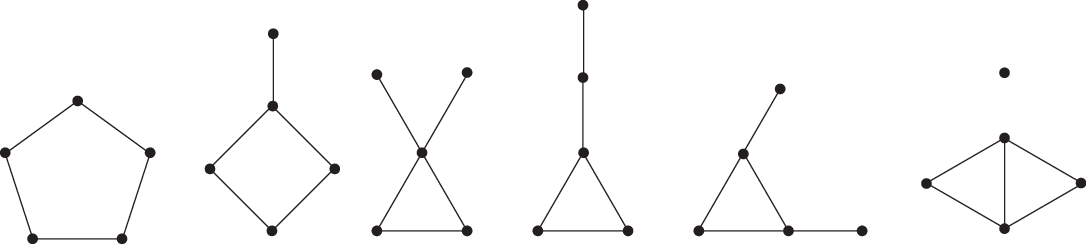}
\caption{All the possible configuration\jm{s} of five points and five edges with a cycle (cycles)}
\label{five_pts}
\end{center}
\end{figure}
This observation is generalized as the following lemma.

\newcommand{\degree}{\mathop{\mathrm{deg}}\nolimits}

\begin{lemma}\label{lemma_valency}
Suppose $G=(V,E)$ is a simple graph with $n$ vertices and $n$ edges. Then \jm{the} number of \jm{un}oriented open 2-paths in $G$ is greater than or equal to $n$,
and equality holds if and only if the valencies of all vertices are equal to two.
\end{lemma}

\begin{proof}
We denote the valency of $v \in V$ by $\degree v$. Then $\sum_{v\in V} \degree v = 2 \# E =2n$.
\jm{Each} open 2-path \jm{has a} midpoint, and each vertex $v$ can be a midpoint of $\binom{\degree v}{2}$ open 2-paths.
Therefore, the number of open 2-paths in $G$ is $\sum_{v\in V} \binom{\degree v}{2}$.
Since the function $\binom{s}{2}=\frac{s(s-1)}{2}$ is \jm{strictly} convex, $\sum_{v\in V} \binom{\degree v}{2}$ gets its minimum
only if $\degree v = 2$ for all $v$, and the minimum value is $\sum_{v\in V} \binom{\degree v}{2} = n$.
\end{proof}

Case $n=5$. 
Put $b=\de_2=\big(1+\sqrt5\,\big)/2$. Note that $1$ and $b$ are rationally independent. 
Assume a finite metric space $X$ has the same formal magnitude as the regular pentagon \jm{$R_5$}. Then by \eqref{fmrp} we have 
\[
m_X(q)=5-10q-10q^b+20q^2+40q^{1+b}+20q^{2\jm{b}}+\dots,
\]
which implies that $\#X=5$ and that $X$ has five edges of length $1$ and another five edges of length $b$. 
It follows that $X$ has ten {\em closed} $2$-paths of length $2=1+1$ and another ten {\em closed} $2$-paths of length $2b=b+b$. From the coefficients of $q^2$ and $q^{2b}$ we know that $X$ has ten {\em open} $2$-paths of length $2=1+1$ and another ten {\em open} $2$-paths of length $2b=b+b$.

First consider only edges of length $1$. From Lemma \ref{lemma_valency}, all the valencies are two, so they form a pentagon.
The same statement holds for length $b$ edges. (The pentagon with side length $b$ should be star-shaped.) 
Therefore $X$ is isometric to the regular pentagon \jm{$R_5$}.

\smallskip
Case $n=7$. 
Put $b=\de_2\approx 1.80$ and $d=\de_3\approx 2.24$. 
Assume a finite metric space $X$ has the same formal magnitude as the regular \jm{heptagon $R_7$}. Then by \eqref{fmrp} we have 
\[
\begin{array}{rcl}
m_X(q)&=&\displaystyle 7-14\left(q+q^b+q^d\right) \\[2mm]
&&+28\left(q^2+q^{2b}+q^{2d}\right)+56\left(q^{1+b}+q^{1+d}+q^{b+d}\right)+\dots,
\end{array}
\]
which implies that $\#X=7$. 
Since 
\[
c_1+c_2b+c_3d=c_1'+c_2'b+c_3'd, \quad c_1,\dots, c_3'\in\NN\cup\{0\}, \> c_1'+c_2'+c_3'\le2,
\]
implies $c_1=c_1', c_2=c_2'$ and $c_3=c_3'$, the coefficients 
of $m_X(q)$ show that $X$ has seven edges of length $1$, seven edges of length $b$ and seven edges of length $d$ and that the number of unoriented open $2$-paths of length $2$, $2b$ and $2d$ are all seven. 

First consider only edges of length $1$. 
Since the numbers of the vertices, length $1$ edges and length $2$ unoriented open $2$-paths are all seven. 
By Lemma \ref{lemma_valency}, the possible configuration of length $1$ edges is either a \jm{heptagon} or a disjoint union of a triangle and a square. 
Next consider the length $d$ edges. The same argument shows that the possible configuration of length $d$ edges is either a \jm{heptagon} or a disjoint union of a triangle and a square. 

Assume the configuration of length $1$ edges is a disjoint union of a triangle $\triangle$ and a square $\square$. 
Since a triangle with side lengths $1,1$ and $d$ does not satisfy the triangle inequality, any length $d$ edge can only connect a vertex of $\triangle$ to a vertex of $\square$. Then it is impossible to form a \jm{heptagon} or a triangle with length $d$ edges, which is a contradiction. 

Therefore the only possible configuration of length $1$ edges is a \jm{heptagon}. 
Let it be $P_0P_1\dots P_6$, where the points are assumed to be in this order. 
The triangle inequality indicates that any length $d$ edge connects $P_i$ and $P_{i\pm 3}$, where the subscripts are considered modulo $7$. 
It follows that the configuration of length $d$ edges is a (star-shaped) \jm{heptagon}. 
Joining the remaining pairs of vertices by length $b$ edges, we obtain a regular \jm{heptagon} $R_7$, which completes the proof of (1). 

\smallskip
(2) Assume a finite metric space $X$ has the same formal magnitude as the $4$-cycle graph $C_4$. Then 
\begin{equation}\label{mxc4}
m_X(q)=4-8q+12q^2-16q^3+\dots,
\end{equation}
which implies that $\#X=4$ and that $X$ has four edges of length $1$. 
The only possible configuration of length $1$ edges is one of the two illustrated in Figure \ref{four_pts}. 
In either case, the lengths $\ell$ of the remaining two edges are both $2$ since from \eqref{mxc4} we know $2\le\ell$, whereas the triangle inequality implies $\ell\le2$. Then the formal magnitude of the right of Figure \ref{four_pts} is equal to $4-8q+14q^2+\dots$, which is not equal to $m_{C_4}(q)$. 
Remark that the difference of the coefficients of $q^2$ is the consequence of the difference of the number of open $2$-paths. 

The case when $n=5$ can be proved in the same way by Figure \ref{five_pts} from $\displaystyle m_{C_5}(q)=5-10q+10q^2+\dots$. 

\smallskip
(3) is the consequence of Proposition\jm{s} \ref{prop_isomer_eg} and \jm{\ref{prop_isomer_mag}}. 
\end{proof}

We remark that the edge lengths $\de_1, \dots, \de_{\lfloor n/2\rfloor}$ of the regular $n$-gon given by \eqref{delta} are not always rationally independent. 
For example, when $n=3k$ $(k\in\NN)$ since 
\[
\sin\frac{(k+j)\pi}{3k}-\sin\frac{(k-j)\pi}{3k}=2\cos\frac\pi3\sin \frac{j\pi}{3k}=\sin \frac{j\pi}{3k},
\]
we have 
\[
\de_j+\de_{k-j}=\de_{k+j} \quad\mbox{ if } \quad  1\le j\le \left\lfloor\frac{k-1}2\right\rfloor.
\]
This makes it difficult to determine the distance multiset, i.e. the multiset of edge lengths 
and what kind and how many paths there are from the coefficients of the formal magnitude. 

\subsection{Identification by magnitude homology}

\begin{theorem}
\begin{enumerate}
\item 
The regular $n$-gon $R_n$ and the isomer $R_n'$ constructed in Theorem \ref{thm_homometric} for even $n\ge6$ and \ref{isomer_Rn_odd} (1) for odd $n\ge\jm{11}$ have isomorphic magnitude homology groups. 
\item The $n$\jm{-}cycle graph $C_n$ and the isomer $C_n'$ constructed in Theorem \ref{thm_homometric} for even $n\ge6$ have non-isomorphic magnitude homology groups.
\end{enumerate}
\end{theorem}

\begin{proof}
(1) Case 1. Suppose $n$ is an even number greater than or equal to $6$. 

First we give a bijection $\varphi_k^l\colon\mathcal{P}_k^l(R_n)\to\mathcal{P}_k^l(R_n')$. 
Let $\gamma=(P_{i_0},P_{i_1},\dots,P_{i_k})$ be a $k$-path in $R_n$ with length $l$. 
Suppose 
\[\varphi_k^l(\gamma)=\left(Q^0_{j_0}, Q^1_{j_1}, \dots, Q^k_{j_k}\right), \quad Q^m=A \>\>\mbox{ or } \>\>B, \>\> 0\le j_m\le n/2-1 \quad (0\le m\le k). 
\]
We define $j_m$ and $Q^m$ inductively. 

(i) For $m=0$, put 
\[\left\{\begin{array}{lll}
j_0=i_0, &\>\> Q^0=A &\quad \mbox{if} \quad 0\le i_0\le n/2-1, \\[2mm]
j_0=i_0-n/2, &\>\> Q^0=B &\quad \mbox{if} \quad n/2\le i_0\le n-1. 
\end{array}\right.\]

(ii) Assume $j_\mu$ and $Q^\mu$ are determined for $0\le\mu\le m-1$. We define $j_m$ and $Q^m$ as follows. 

Case 1. Suppose $n/2$ is even, say $n=4p$ $(p\ge2)$. Then 
\[\left\{\begin{array}{lll}
j_m=j_{m-1}+(i_m-i_{m-1}) \quad{\rm mod} \> 2p, & \>\> Q^m=Q^{m-1}  & \quad \mbox{if} \quad 1\le i_m-i_{m-1}\le p \quad{\rm mod} \> n \>\>\mbox{ or } \\[1mm]
&& \quad \phantom{\mbox{if}} \quad  3p+1\le i_m-i_{m-1}\le 4p-1 \quad{\rm mod} \> n, \\[2mm]
j_m=j_{m-1}+(i_m-i_{m-1})-\jm{p} \quad{\rm mod} \> 2p, & \>\> Q^m\ne Q^{m-1}  & \quad \mbox{if} \quad p+1\le i_m-i_{m-1}\le 3p \quad{\rm mod} \> n. 
\end{array}\right.\]

Case 2. Suppose $n/2$ is odd, say $n=4p+2$ $(p\ge1)$. Then 
\[\left\{\begin{array}{lll}
j_m=j_{m-1}+(i_m-i_{m-1}) \quad{\rm mod} \> 2p+1, & \>\> Q^m=Q^{m-1}  & \quad \mbox{if} \quad 1\le i_m-i_{m-1}\le p \quad{\rm mod} \> n \>\>\mbox{ or } \\[1mm]
&& \quad \phantom{\mbox{if}} \quad  3p+2\le i_m-i_{m-1}\le 4p+1 \quad{\rm mod} \> n, \\[2mm]
j_m=j_{m-1}+(i_m-i_{m-1}) \quad{\rm mod} \> 2p+1, & \>\> Q^m\ne Q^{m-1}  & \quad \mbox{if} \quad p+1\le i_m-i_{m-1}\le 3p+1 \quad{\rm mod} \> n. 
\end{array}\right.\]

The bijection $\varphi_k^l$ can be extended to an isomorphism from $MC^l_k(R_n)$ to $MC^l_k(R_n')$, which we denote by the same letter $\varphi_k^l$. 

\smallskip
Next we show that $\varphi_k^l$ induces an isomorphism between the magnitude homology groups, $MH^l_k(R_n)$ and $MH^l_k(R_n')$. 
The proofs of Lemma \ref{lem_small_m-gon} and Theorem \ref{thm_homometric} imply that the triangle inequalities that appear in the isomer $R_n'$ come from the circular triangle inequalities \eqref{triangle_inequality_delta} of $\de_i$'s. 
Since any triangle inequality that appears in $R_n$ is strict, the same assertion holds for the $R_n'$. 
Therefore all the boundary operators $\partial$ are zero maps, therefore ${\rm Ker}\>\partial^l_k=MC^l_k(R_n)\cong MC^l_k(R_n')={\rm Ker}\>{\partial^l_k}'$, which implies ${\varphi_k^l}_\ast\colon MH^l_k(R_n) \stackrel{\cong}{\to}  MH^l_k(R_n')$.

\smallskip
Case 2. Suppose $n$ is an odd number greater than or equal to $11$. Let $n=2p+1$. 

Define a bijection $\psi_k^l\colon\mathcal{P}_k^l(R_n)\to\mathcal{P}_k^l(R_n')$ by 
\[
\psi_k^l\colon (P_{i_0},P_{i_1},\dots,P_{i_k})\mapsto(P_{i_0'}',P_{i_1'}',\dots,P_{i_k'}'),
\]
where $i_0'=i_0$ and $i_m'$ $(m\ge1)$ is defined inductively by the following. 
\begin{enumerate}
\item[(i)] $i_m'-i_{m-1}' \>\> ({\rm mod}. \>\> n)\ge \jm{p}+1$ if and only if $i_m-i_{m-1} \>\> ({\rm mod}. \>\> n)\ge \jm{p}+1$, 
\item[(ii)] $\displaystyle 
|i_m'-i_{m-1}'|_n=\begin{cases}
|i_m-i_{m-1}|_n &\quad \mbox{if } \>\>\> |i_m-i_{m-1}|_n\le \jm{p}-2, \\
\jm{p} &\quad \mbox{if } \>\>\> |i_m-i_{m-1}|_n= \jm{p}-1, \\
\jm{p}-1 &\quad \mbox{if } \>\>\> |i_m-i_{m-1}|_n= \jm{p}. \\
\end{cases}
$
\end{enumerate}
The rest can be proved in the same way as in Case 1. 
The strict triangle inequality follows from Lemma \ref{lem_de_1_de_k-2}. 

\smallskip
(2) Let $C_A$ and $C_B$ denote $X_A'$ and $X_B'$ in the proof of Theorem \ref{thm_homometric}. 
Put $l_0=\lceil n/4 \rceil$. Then $l_0={\rm dist}(C_A,C_B)$. 
There are $n/2$ edges joining $C_A$ and $C_B$ with length $l_0$, $\{A_iB_i\}_{i=0}^{n/2-1}$, if $n/2$ is even, and $n$ edges joining $C_A$ and $C_B$ with length $l_0$, $\{A_iB_{i+p}, A_iB_{i+p+1}\}_{i=0}^{n/2-1}$, where $p=(n-2)/4$ and subscripts are considered modulo $n/2$, if $n/2$ is odd. 
It follows that 
\[
MH^{l_0}_1(C_n')\cong\left\{\begin{array}{ll}
\ZZ^{n} & \quad \mbox{if $n/2$ is even}, \\[1mm]
\ZZ^{2n} & \quad \mbox{if $n/2$ is odd}, 
\end{array}
\right.
\]
whereas $MH^{l_0}_1(C_n)\cong 0$ for any even $n$ with $n\ge6$. 

\end{proof}

\begin{remark}\rm 
The proof of (1) does not work for $n=9$ as stated. 
In $R_9$, all the triangle inequalities are strict, whereas its isomer $R'_9$ \jm{(or $R''_9$)} constructed as in Proposition \ref{isomer_Rn_odd} (1) \jm{(or (2) respectively)} 
contains a triangle with side lengths $\de_1, \de_2$ and $\de_4$ satisfying $\de_1+\de_2=\de_4$. 
Consequently, when computing the magnitude homology of $R'_9$, a nonzero operator $\partial$ appears, hence the argument of the proof of (1) no longer applies as it stands. 

Incidentally, every isomer of $R_9$ necessarily contains a triangle with side lengths $\de_1, \de_2$ and $\de_4$. 
This is because if, in addition to the triangle inequalities, we were to forbid the triple $(\de_1, \de_2, \de_4)$, the resulting condition would be equivalent to the one that an isomer of $C_9$ must satisfy, although no isomer of $C_9$ exists as shown in Proposition \ref{non-exists_isomer_cycle}. 
\end{remark}

Gu gave an example of pairs of {\sl graphs} with the same magnitude and non-isomorphic magnitude homology with $16$ (and $20$) vertices (\cite{G} Appendix A). 
$C_6$ and $C_6'$ give a pair of metric spaces with the same property with \jm{fewer} vertices by removing the condition that the spaces be graphs.

\noindent
Hiroki Kodama

\noindent
International Institute for Sustainability with Knotted Chiral Meta
Matter, Hiroshima University,

\noindent
1-3-1 Kagamiyama, Higashi-Hiroshima City, Hiroshima, 739-8531, Japan.

\noindent
Center for Interdisciplinary Theoretical and Mathematical Sciences, RIKEN,

\noindent
2-1 Hirosawa, Wako-shi, Saitama, 351-0198, Japan.

\noindent
E-mail: kodamahiroki@gmail.com

\bigskip

\noindent
Jun O'Hara

\noindent
Department of Mathematics and Informatics, Faculty of Science, 
Chiba University

\noindent
1-33 Yayoi-cho, Inage, Chiba, 263-8522, JAPAN.  

\noindent
E-mail: ohara@math.s.chiba-u.ac.jp


\begin{thebibliography}{OS} 
\bibitem{BK}M.~Boutin and G.~Kemper, {\em On reconstructing $n$-point configurations from the distribution of distances or areas.} Adv. in Appl. Math. 32 (2004), 709\,--\,735.

\bibitem{Bry} J.-L. Brylinski, {\em The beta function of a knot.} Internat. J. Math. {\bf 10} (1999) 415\,--\,423. 

\bibitem{Cae}T.~Caelli, {\em On generating spatial configurations with identical interpoint distance distributions.} in: Combinatorial Mathematics, VII (Proc. Seventh Australian Conf., Univ. Newcastle, Newcastle, 1979) in: Lecture Notes in Math., vol. 829, Springer, Berlin (1980)  69\,--\,75.

\bibitem{FV}E. J.~Fuller and M.K.~Vemuri. {\em The Brylinski Beta Function of a Surface.}  Geometriae Dedicata 179 (2015)  153\,--\,160, doi:10.1007/s10711-015-0071-y.

\bibitem{GG}H.~Gimperlein and M.~Goffeng, {\em Riesz energies and the magnitude of manifolds.} Proc. Amer. Math. Soc. 153 (2025), 3173\,--\,3184.

\bibitem{GGL}H.~Gimperlein, M.~Goffeng and N.~Louca, {\em The magnitude and spectral geometry}.  to appear in Amer. J. Math. arXiv:2201.11363. 

\bibitem{G}Y.~Gu, {\em Graph magnitude homology via algebraic Morse theory}. arXiv:1809.07240.

\bibitem{H}F.~Harary, {\em Graph Theory}, Addison Wesley, 1969. 


\bibitem{HW}R.~Hepworth and S.~Willerton, {\em Categorifying the magnitude of a graph}, Homol. Homotopy Appl. 19 (2017), 31\,--\,60. 

\bibitem{L13}T.~Leinster, {\em The magnitude of metric spaces}. Doc. Math. 18 (2013), 857\,--\,905.

\bibitem{L17}T.~Leinster, {\em The magnitude of a graph}. Math. Proc. Camb. Phil. Soc. 166 (2017) 247\,--\,264.

\bibitem{LS}T.~Leinster and M.~Shulman, {\em Magnitude homology of enriched categories and metric spaces}, Alg. Geom. Topol. 21 (2021), 2175\,--\,2221. 

\bibitem{MC}C.~L.~Mallows and J.~M.~C.~Clark, {\em Linear-Intercept Distributions Do Not Characterize Plane Sets.} J. Appl. Prob. 7 (1970), 240\,--\,244.


\bibitem{Me}M.~Meckes, {\em Magnitude, diversity, capacities and dimensions of metric spaces}, Potential Anal. 42 (2015), 549\,--\,572.


\bibitem{O24}J.~O'Hara, {\em Magnitude function identifies generic finite metric spaces}, to appear in Discrete Analysis, arXiv:2401.00786. 

\bibitem{OS2}J.~O'Hara and G.~Solanes, {\em Regularized Riesz energies of submanifolds.}  Math. Nachr. 291 (2018), 1356-1373, DOI: 10.1002/mana.201600083. 


\bibitem{RS}J.~Rosenblatt and P.D.~Seymour, {\em The structure of homometric sets.} 
SIAM J. Algebraic Discrete Methods 3 (1982),  343\,--\,350.

\bibitem{LSS}P.~Lemke, S.S.~Skiena, and W.D.~Smith, {\em Reconstructing sets from interpoint distances.} Discrete and computational geometry, Algorithms Combin. 25, Springer-Verlag, Berlin, 2003, 597\,--\,631.

\bibitem{W}P.~Waksman, {\em Polygons and a conjecture of Blaschke's}, Adv. Appl. Prob. 17 (1985), 774\,--\,793.

\end{thebibliography}
\end{document}